\documentclass[10pt, a4paper]{amsart}
\usepackage{amssymb, amsmath, amsfonts, verbatim}
\usepackage{color}
\newtheorem{lemma}[equation]{Lemma}
\newtheorem{theorem}[equation]{Theorem}

\newtheorem{conjecture}{Conjecture}
\newtheorem{question}{Question}
\newtheorem{problem}{Problem}
\newtheorem{definition}[equation]{Definition}
\newtheorem{cor}[equation]{Corollary}
\theoremstyle{remark}

\newtheorem{example}[equation]{Example}

\title{Strongly Real Beauville Groups}
\date{}

\author{Ben Fairbairn}
\address{Ben Fairbairn, Department of Economics, Mathematics and Statistics, Birkbeck, University of London, Malet Street, London WC1E 7HX, United Kingdom}
\email{b.fairbairn@bbk.ac.uk}

\begin{document}
\maketitle

\begin{abstract}
A strongly real Beauville group is a Beauville group that defines a
real Beauville surface. Here we discuss efforts to find examples of
these groups, emphasising on the one extreme finite simple groups
and on the other abelian and nilpotent groups. We will also discuss
the case of characteristically simple groups and almost simple groups. \emph{En route} we shall
discuss several questions, open problems and conjectures as well as giving several new examples of infinite families of strongly real Beauville groups.
\end{abstract}

\section{Introduction}

We first issue an apology/assurance. It is the nature of Beauville
constructions that this article is likely to be of interest to both
geometers and group theorists. The author is painfully aware of
this. As a consequence there will be times when we make statements
that may seem obvious or elementary to the group theorist but may
seem quite surprising to the geometer.

We begin with the usual definitions to establish notation and
terminology.

\begin{definition}
A surface $\mathcal{S}$ is a \textbf{Beauville surface of unmixed
type} if
\begin{itemize}
\item the surface $\mathcal{S}$ is isogenous to a higher product, that is,
$\mathcal{S}\cong(\mathcal{C}_1\times\mathcal{C}_2)/G$ where
$\mathcal{C}_1$ and $\mathcal{C}_2$ are algebraic curves of genus at
least $2$ and $G$ is a finite group acting faithfully on
$\mathcal{C}_1$ and $\mathcal{C}_2$ by holomorphic transformations
in such a way that it acts freely on the product
$\mathcal{C}_1\times\mathcal{C}_2$ and
\item each $\mathcal{C}_i/G$ is isomorphic to the projective line $\mathbb{P}_1(\mathbb{C})$, and the covering map $\mathcal{C}_i\rightarrow\mathcal{C}_i/G$ is
ramified over three points.
\end{itemize}
\end{definition}

What makes these surfaces so easy to work with is the fact that the
definition above can be translated into purely group theoretic terms --- the following definition imposes equivalent conditions on the group $G$.

\begin{definition}\label{DagDef}
Let $G$ be a finite group. Let $x,y\in G$ and let
$$\Sigma(x,y):=\bigcup_{i=1}^{|G|}\bigcup_{g\in G}\{(x^i)^g,(y^i)^g,((xy)^i)^g\}.$$

An \textbf{unmixed Beauville structure} for $G$ is a pair of generating
sets of elements $\{\{x_1,y_1\},\{x_2,y_2\}\}\subset G\times G$ such that
$\langle x_1,y_1\rangle=\langle x_2,y_2\rangle=G$ and
\begin{equation}
\Sigma(x_1,y_1)\cap\Sigma(x_2,y_2)=\{e\}\tag{$\dagger$}
\end{equation}
where $e$ denotes the identity element of $G$. If $G$ has
a Beauville structure, then we say that $G$ is a \textbf{Beauville group}.
Furthermore we say that the structure has \textbf{type}
$((o(x_1),o(y_1),o(x_1y_1)),(o(x_2),o(y_2),o(x_2y_2))).$
\end{definition}

In the author's experience, upon seeing the above definition, group
theorists often retort ``why record just the orders of the elements
and not precisely which classes the elements belong to?" Determining
precisely which class an element belongs to is much harder than determining its order. Furthermore, in practice, when ensuring
that a set of elements satisfies condition ($\dagger$) the easiest
way often is to show that $o(x_1)o(y_1)o(x_1y_1)$ is coprime to
$o(x_2)o(y_2)o(x_2y_2)$. This simple observation has been used to
great effect by several authors --- see
\cite{FairbairnExceptional,FairbairnMagaardParker,FuertesJones11,GuralnickMalle12}
among others. Furthermore, the type alone encodes substantial
amounts of geometric information: the Riemann-Hurwitz formula
\[
g(\mathcal{C}_i)=1+\frac{|G|}{2}\bigg(1-\frac{1}{o(x_i)}-\frac{1}{o(y_i)}-\frac{1}{o(x_iy_i)}\bigg)
\]
tells us the genus of each of the curves used to define the surface
$\mathcal{S}$. Indeed, whilst some groups have generating pairs that by the above formula define surfaces with the property that $g(\mathcal{C})\leq1$, condition $(\dagger)$ ensures that for each $i$ we have that $g(\mathcal{C}_i)\geq2$. Furthermore, a theorem of Zeuthen-Segre also gives
us the Euler number of the surface $\mathcal{S}$ since
\[
e(\mathcal{S})=4\frac{(g(\mathcal{C}_1)-1)(g(\mathcal{C}_2)-1)}{|G|},
\]
which in turn gives us the holomorphic Euler-Poincar\'{e}
characteristic of $\mathcal{S}$ from the relation
$4\chi(\mathcal{S})=e(\mathcal{S})$ --- see \cite[Theorem
3.4]{Catanese00}.

In light of the above, we make the following non-standard definition
which will be of use in what follows.

\begin{definition}
We say that a Beauville structure $\{\{x_1,y_1\},\{x_2,y_2\}\}$ is
\textbf{coprime} if $o(x_1)o(y_1)o(x_1y_1)$ and
$o(x_2)o(y_2)o(x_2y_2)$ are coprime.
\end{definition}

Given any complex surface $\mathcal{S}$ it is natural to consider
the complex conjugate surface $\overline{\mathcal{S}}$. In
particular, it is natural to ask whether the surfaces are biholomorphic.

\begin{definition}
Let $\mathcal{S}$ be a complex surface. We say that $\mathcal{S}$ is
\textbf{real} if there exists a biholomorphism
$\sigma:\mathcal{S}\rightarrow\overline{\mathcal{S}}$ such that
$\sigma^2$ is the identity map.
\end{definition}

As is often the case with Beauville surfaces, the above geometric
condition can be translated into purely group theoretic terms.

\begin{definition}
Let $G$ be a Beauville group. We say that $G$ is
\textbf{strongly real} if there exists a Beauville structure
$X=\{\{x_1,y_1\},\{x_2,y_2\}\}$ such that there exists an automorphism $\phi\in
\mbox{Aut}(G)$ and elements $g_i\in G$ for $i=1,2$ such
that
$$g_i\phi(x_i)g_i^{-1}=x_i^{-1}\mbox{ and }g_i\phi(y_i)g_i^{-1}=y_i^{-1}$$
for $i=1,2$. In this case we also say that the Beauville structure $X$ is a strongly real Beauville structure.
\end{definition}

In practice we can always replace one generating pair by some
conjugate of it and so we can take $g_1=g_2=e$ and often this is
what is done in practice.

In \cite{BauerCataneseGrunewald05} Bauer, Catanese and Grunewald
show that a Beauville surface is real if, and only if, the
corresponding Beauville group and structure are strongly real.

\begin{example}
In \cite{Catanese03} Catanese classified the abelian Beauville
groups by proving the following.

\begin{theorem}\label{ab}
If $G$ is an abelain group, then $G$ is a Beauville group if, and only
if, $G\cong\mathbb{Z}_n\times\mathbb{Z}_n$ where gcd($n$,6)=1 and
$\mathbb{Z}_n$ denotes the cyclic group of order $n>1$.
\end{theorem}

This theorem immediately gives us the following.

\begin{cor}\label{abCor}
Every abelian Beauville group is a strongly real Beauville group making any Beauville structure for these groups strongly real.
\end{cor}

\begin{proof}
If $H$ is an abelian group, then the map $H\rightarrow H$,
$x\mapsto -x$ is an automorphism.
\end{proof}
\end{example}

More recent (and group theoretic) motivation comes from the
following. The absolute Galois group
Gal$(\overline{\mathbb{Q}}/\mathbb{Q})$ is very poorly understood.
Indeed, The Inverse Galois Problem --- arguably the hardest open
problem in algebra today --- forms just one small part of efforts to
understand Gal$(\overline{\mathbb{Q}}/\mathbb{Q})$ (it amounts to
showing that every finite group arises as the quotient of Gal$(\overline{\mathbb{Q}}/\mathbb{Q})$ by a topologically closed
normal subgroup). When confronted with the task of understanding a
group it is natural to consider an action of the group on some set.
The group Gal$(\overline{\mathbb{Q}}/\mathbb{Q})$ acts on the set of
Beauville surfaces thanks to Grothendieck's theory of Dessins
d'enfants (``children's drawings"). See \cite[Section
11]{JonesSurvey} for a more detailed discussion of this and related
matters.

Henceforth we shall use the standard {\sc Atlas} notation for group
theoretic concepts (aside from occasional deviations to minimise confusion with geometric concepts) as described in some detail in the introductory
sections of \cite{ATLAS}. In particular, given two groups $A$ and $B$ we use the following notation.
\begin{itemize}
\item We write $A\times B$ for the direct product of $A$ and $B$, that is, the group whose members are ordered pairs $(a,b)$ with $a\in A$ and $b\in B$ such that for $(a,b),(a',b')\in A\times B$ we have the multiplication $(a,b)(a',b')=(aa',bb')$. Given a positive integer $k$ we write $A^k$ for the direct product of $k$ copies of $A$.
\item We write $A.B$ for the extension of $A$ by $B$, that is, a group with a normal subgroup isomorphic to $A$ whose quotient is $B$ (such groups are not necessarily direct products - for instance SL$_2(5)$=2.PSL$_2(5)$). 
\item We write $A:B$ for a semi-direct product of $A$ and $B$, also known as a split extension $A$ and $B$, that is, there is a homomorphism $\phi\colon B\rightarrow Aut(A)$ with elements of this group being ordered pairs $(b,a)$ with $a\in A$ and $b\in B$ such that for $(b,a),(b',a')\in A:B$ we have the multiplication $(b,a)(b,a)=(bb',a^{\phi(b')}a')$.
    \item We write $A\wr B$ for the wreath product of $A$ and $B$, that is, if $B$ is a permutation group on $n$ points then we have the split extension $A^n:B$ with $B$ acting in a way that permutes the $n$ copies of $A$.
\end{itemize}

In several places we shall refer to `straightforward computations'
or calculations that readers can easily reproduce for themselves.
On these occasions either of Magma \cite{Magma} or GAP \cite{GAP}
can easily be used to do this.

In Section 2 we will discuss the finite simple groups and in particular a conjecture of Bauer, Catanese and Grunewald concerning which of these groups are strongly real Beauville groups. In Sections 3 our attention turns to the characteristically simple groups and in particular the recent work of Jones which we push further in the cases of the symmetric and alternating groups in Section 4. We go on in Section 5 to discuss which of the almost simple groups are strongly real Beauville groups. Finally, in Section 6 we briefly discuss nilpotent groups and $p$-groups.

\section{The Finite Simple Groups}\label{FSG}

Naturally, a necessary condition for being a strongly real Beauville
group is being a Beauville group. Furthermore, a necessary condition
for being a Beauville group is being 2-generated: we say that a group $G$ is 2-generated if there exist two elements $x,y\in G$ such that $\langle x,y\rangle=G$. It is an
easy exercise for the reader to show that the alternating groups
$A_n$ for $n\geq3$ are 2-generated. In \cite{Steinberg} Steinberg
proved that the simple groups of Lie type are 2-generated and in \cite{AG}
Aschbacher and Guralnick showed that the sporadic simple groups are
2-generated. We thus have that all of the non-abelian finite simple
groups are 2-generated making them natural candidates for Beauville
groups. This lead  Bauer, Catanese and Grunewald to conjecture that
aside from $A_5$, which is easily seen to not be a Beauville group,
every non-abelian finite simple group is a Beauville group - see
\cite[Conjecture 1]{BauerCataneseGrunewald05} and \cite[Conjecture
7.17]{BauerCataneseGrunewald06}. This suspicion was later proved
correct
\cite{FairbairnMagaardParker,FairbairnMagaardParker2,GarionLarsenLubotzky12,GuralnickMalle12},
indeed the full theorem proved by the author, Magaard and Parker in
\cite{FairbairnMagaardParker} is actually a more general statement
about quasisimple groups (recall that a group $G$ is quasisimple if it is generated by its commutators and the quotient by its center $G/Z(G)$ is a simple group.).

Having found that almost all of the non-abelian finite simple groups
are Beauville groups, it is natural to ask which of the non-abelian
finite simple groups are strongly real Beauville groups. In
\cite[Section 5.4]{BauerCataneseGrunewald05} Bauer, Catanese and
Grunewald wrote\\

\begin{quotation}
``There are 18 finite simple nonabelian groups of order $\leq15000$.
By computer calculations we have found strongly [real] Beauville
structures on all of them with the exceptions of $A_5$, PSL$_2$(7),
$A_6$, $A_7$, PSL$_3$(3), U$_3$(3) and the Mathieu group M$_{11}$."
\end{quotation}\hspace{10mm}

On the basis of these computations they conjectured that all but
finitely many non-abelian finite simple groups are strongly real
Beauville groups. Several authors have worked on this and many
special cases are now known to be true.
\begin{itemize}
\item In \cite{FuertesGD10} Fuertes and Gonz\'{a}lez-Diez showed
that the alternating groups $A_n$ ($n\geq7$) and the symmetric
groups $S_n$ ($n\geq5$) are strongly real Beauville groups by
explicitly writing down permutations for their generators and the
automorphisms and applying some of the classical theory of
permutation groups to show that their elements had the properties
they claimed. Subsequently the alternating group A$_6$ was also shown to be a strongly real Beauville group.

\item In \cite{FuertesJones11} Fuertes and Jones prove that the
simple groups $PSL_2(q)$ for prime powers $q>5$ and the quasisimple
groups $SL_2(q)$ for prime powers $q>5$ are strongly real Beauville
groups.  As with the alternating and symmetric groups, these results
are proved by writing down explicit generators, this time combined
with a celebrated theorem usually (but historically inaccurately)
attributed to Dickson for the maximal subgroups of $PSL_2(q)$.
General lemmas for lifting Beauville structures from a group to its covering
groups are also used.

\item Settling the case of the sporadic simple groups makes no
impact on the above conjecture, there being only 26 of them.
Nonetheless, for reasons we shall return to below, in
\cite{FairbairnExceptional} the author determined which of the
sporadic simple groups are strongly real Beauville groups, including
the `27$^{th}$ sporadic simple group', the Tits group $^2\mbox{F}_4(2)'$.
Of all the sporadic simple groups only the Mathieu groups M$_{11}$ and M$_{23}$ are not strongly real.
For all of the other sporadic groups smaller than the Baby Monster
group $\mathbb{B}$ explicit words in the `standard generators'
\cite{Wilson} for a strongly real Beauville structure are given. (For those unfamiliar with standard generators, we will describe these in Section \ref{AS}.) For
the Baby Monster group $\mathbb{B}$ and Monster group $\mathbb{M}$
character theoretic methods are used.
\end{itemize}

As we can see from the above bullet points, several of the groups
that Bauer, Catanese and Grunewald could not find strongly real
Beauville structures for do indeed have strongly real Beauville
structures. In particular, we note that the group PSL$_2(9)\cong$ A$_6$
is in fact strongly real.

Using the results mentioned above, combined with unpublished calculations, the author has pushed
Bauer, Catanese and Grunewald's original computations to every
non-abelian finite simple group of order at most $100\,000\,000$
and, as we noted above, several much larger ones in
\cite{FairbairnExceptional}. Many of the smaller groups seemed to
require the use of outer automorphisms to make their Beauville
structures strongly real, which explains much of the above
difficulty in finding strongly real Beauville structures in certain groups. Slightly larger groups had enough conjugacy classes for
inner automorphisms to be used instead. Consequently, it seems that
`small' non-abelian finite simple groups fail to be strongly real if
they have too few conjugacy classes (as is the case with $A_5$ and
as we would intuitively expect) or if they have no outer automorphisms
--- a phenomenon that is extremely rare. We are thus lead to the following
somewhat stronger conjecture.

\begin{conjecture}\label{simpconj}
All non-abelian finite simple groups apart from $A_5$, M$_{11}$ and
M$_{23}$ are strongly real Beauville groups.
\end{conjecture}

To add further weight to this conjecture we verify this conjecture
for the Suzuki groups $^2B_2(2^{2n+1})$. Let $q=2^{2n+1}$.

\begin{theorem}\label{Suz}
Each of the groups $^2B_2(q)$ has a strongly real Beauville
structure of type $((q-1,q-1,q-1),(d_1,d_2,2))$ where $d_1$ and
$d_2$ are odd and coprime to $q-1$.
\end{theorem}

Throughout the following we shall be using the natural
4-dimensional representation of $^2B_2(q)$ over the field of order
$q$ as described in some detail in \cite[Section 4.2]{WilsonBook}. To prove Thoerem \ref{Suz} we will use knowledge of the maximal subgroups of the
Suzuki groups. The following lemma was proved by Suzuki --- see
\cite[Theorem 4.1]{WilsonBook}. Here we write $E_q$ for the
elementary abelian group of order $q$. Furthermore, by `subfield subgroup' we mean either the subgroup $^2B_2(q_0)$ consisting of matrices whose entries come from a subfield of the field $\mathbb{F}_q$ of order $\mathbb{F}_{q_0}$ where $q_0>1$ divides $q$ or one of its conjugates, those appearing in Lemma \ref{SuzMax}(v) being precisely the maximal subfield subgroups.

\begin{lemma}\label{SuzMax}
If $n > 1$, then the maximal subgroups of $^2B_2(q)$ are (up to
conjugacy).
\begin{enumerate}
\item[(i)] $E_q.E_q:\mathbb{Z}_{q-1}$, the subgroup of lower triangular matrices
\item[(ii)] $D_{2(q-1)}$
\item[(iii)] $\mathbb{Z}_{q+\sqrt{2q}+1}:4$
\item[(iv)] $\mathbb{Z}_{q-\sqrt{2q}+1}:4$
\item[(v)] $^2B_2(q_0)$ where $q = q_0^r$, $r$ is prime and $q_0>2$.
\end{enumerate}
\end{lemma}

From the above the following can easily be deduced.

\begin{lemma}\label{SuzGen}
\begin{enumerate}
\item[(a)] If $x,y\in$ $^2B_2(q)$ are two elements with the
property that $$o(x) = o(y) = o(xy) = q-1,$$ then
$\langle
x,y\rangle=\mathbb{Z}_{q-1}$, $E_q.E_q:\mathbb{Z}_{q-1}$ or
$^2B_2(q)$.
\item[(b)] If
$x,y\in$ $^2B_2(q)$ are two elements such that $o(x)$ and $o(y)$ are
have orders dividing $q\pm\sqrt{2q}+1$ and $o(xy) = 2$, then $\langle
x,y\rangle=$ $^2B_2(q)$ or a subfield subgroup.
\end{enumerate}
\end{lemma}

\begin{proof}[Proof of Theorem \ref{Suz}]

For our first generating pair we consider the following elements of
$^2B_2(q)$ each of which are easily checked to have order 2 by
direct calculation.

\[
t_1:=\left(\begin{array}{cccc}
0&0&0&1\\
0&0&1&0\\
0&1&0&0\\
1&0&0&0
\end{array}\right)\hspace{20mm}
t_2:=\left(\begin{array}{cccc}
0&0&0&\beta^{-1}\\
0&0&\beta^{-2^{n+1}+1}&0\\
0&\beta^{2^{n+1}-1}&0&0\\
\beta&0&0&0
\end{array}\right)
\]

\[
t_3:=\left(\begin{array}{cccc}
1&0&0&0\\
0&1&0&0\\
\alpha^{2^{n+1}}&0&1&0\\
\alpha^2&\alpha^{2^{n+1}}&0&1
\end{array}\right)
\]
where $\alpha$ and $\beta$ are generators of the multiplicative
group $\mathbb{F}_q^{\times}$. The element $x_1=t_1t_2$ has order $q-1$. The characteristic
polynomial of $y_1=t_1t_3$ is
\[
p_1(\lambda)=\lambda^4+\alpha^2\lambda^3+\alpha^{2^{n+2}}\lambda^2+\alpha^2\lambda+1
\]
and if we set
$\gamma:=\beta+\beta^{-1}+\beta^{2^{n+1}-1}+\beta^{1-2^{n+1}}$ then
the characteristic polynomial of $t_1t_2$ is
\[
p_2(\lambda)=\lambda^4+\gamma\lambda^3+(\beta^{2^{n+1}}+\beta^{-2^{n+1}+2}+\beta^{2^{n+1}-2}+\beta^{-2^{n+1}})\lambda^2+\gamma\lambda+1
\]

Comparing $p_1$ with $p_2$ we see that the two polynomials are equal
if we have
\[
\gamma=\alpha^2\mbox{ and }
\]
\[
\beta^{2^{n+1}}+\beta^{-2^{n+1}+2}+\beta^{2^{n+1}-2}+\beta^{-2^{n+1}}=\alpha^{2^{n+2}}=(\alpha^2)^{2^{n+1}}=\gamma^{2^{n+1}}.
\]
Since $a\mapsto a^2$ is an automorphism of our underlying field we
see that the first of these equalities immediately implies the
second if
\[
(\beta+\beta^{-1}+\beta^{2^{n+1}-1}+\beta^{1-2^{n+1}})^{2^{n+1}}=\beta^{2^{n+1}}+\beta^{-2^{n+1}+2}+\beta^{2^{n+1}-2}+\beta^{-2^{n+1}}.
\]
Since
$\beta^{(2^{n+1}-1)2^{n+1}}=(\beta^{2^{2n+1}})^2\beta^{-2^{n+1}}=\beta^{2-2^{n+1}}$
we can choose $\alpha$ and $\beta$ to satisfy the above condition,
so in particular we have that $t_1t_2$ and $t_1t_3$ have the same
characteristic polynomial and thus both have order $q-1$.
Furthermore, these are both inverted by conjugation by $t_1$ since
$t_1$, $t_2$ and $t_3$ all have order 2. Similarly we find that
$t_1t_2t_1t_3$ has characteristic polynomial of the correct form to
have order $q-1$. From Lemma \ref{SuzGen}(a) we see that these
elements generate the group since $x_1$ and $y_1$ are not both contained
in a cyclic subgroup (one of them is diagonal) and by direct
calculation no one-dimensional subspace in the natural module is
preserved by them so there is no proper subgroup containing each of
these elements.

For the second triple we consider the matrices
\[
x_2:=\left(\begin{array}{cccc}
0&0&0&1\\
0&0&1&0\\
0&1&0&\delta^4\\
1&0&\delta^4&\delta^2
\end{array}\right)\hspace{20mm}
y_2:=\left(\begin{array}{cccc}
\epsilon^2&\epsilon^4&0&1\\
\epsilon^4&0&1&0\\
0&1&0&0\\
1&0&0&0
\end{array}\right)
\]
where $\delta,\epsilon\in\mathbb{F}_q$ are chosen so that
$\delta\not=\epsilon$ and these do not have the correct form for
these elements to have order $q-1$. Direct calculation shows that
these elements do not have orders 2 or 4 and that $o(x_2y_2)=2$.
These elements must, therefore, have orders that divide
$q\pm\sqrt{2q}+1$. Furthermore their traces are $\epsilon^2$ and $\delta^2$ which can be chosen to be in no proper subfield since $x\mapsto x^2$ is an automorphism of the field $\mathbb{F}_q$. These elements must therefore generate the group by
Lemma \ref{SuzGen}(b). Further direct calculation shows that
$x_2^{t_1}=x_2^{-1}$ and $y_2^{t_1}=y_2^{-1}$.
\end{proof}

\section{Characteristically Simple Groups}

Another class of finite groups that has recently been studied from
the viewpoint of Beauville constructions, and seems like fertile
ground for providing further examples of strongly real Beauville
groups, are the characteristically simple groups that we define as
follows (the definition commonly given is somewhat different from
that below but in the case finite groups it is equivalent to this).

\begin{definition}
A finite group $G$ is said to be \textbf{characteristically simple}
if $G$ is isomorphic to some direct product $H^k$ where $H$ is a
finite simple groups.
\end{definition}

For example, as we saw in Theorem \ref{ab}, if $p>3$ is prime then
the abelian Beauville groups isomorphic to
$\mathbb{Z}_p\times\mathbb{Z}_p$ are characteristically simple.

Characteristically simple Beauville groups have recently been
investigated by Jones in \cite{Jones1,Jones2} where the following
conjecture is discussed.

\begin{conjecture}
Let $G$ be a finite non-abelian characteristically simple group.
Then $G$ is a Beauville group if and only if it is a $2$-generator
group not isomorphic to $A_5$.
\end{conjecture}

In particular, the main results of \cite{Jones1,Jones2} verify this
conjecture in the cases where $H$ is any of the alternating groups;
the linear groups $PSL_2(q)$ and $PSL_3(q)$; the unitary groups
$PSU_3(q)$; the Suzuki groups $^2B_2(2^{2n+1})$; the small Ree
groups $^2G_2(3^{2n+1})$ and the sporadic simple groups.

For large values of $k$, the group $H^k$ will not be 2-generated
despite the fact that $H$ will be as discussed in Section \ref{FSG}.
The values of $k$ for which $H^k$ is 2-generated can be surprisingly
large. For example, a special case of the results alluded to in the
previous paragraph is the somewhat amusing fact that
$$A_5\times A_5\times A_5\times A_5\times A_5\times A_5\times A_5\times A_5\times A_5\times A_5\times$$ $$A_5\times A_5\times A_5\times A_5\times A_5\times A_5\times
A_5\times A_5\times A_5$$ is a Beauville group, despite the fact
that $A_5$ itself is not a Beauville group.

In general, the full automorphism group of $H^k$ will be the wreath
product $\mbox{Aut}(H)\wr S_k$ where $S_k$ is the $k^{th}$ symmetric group
acting on the product by permuting the groups $H$. This bounteous
supply of automorphisms makes it likely that characteristically
simple Beaville groups are in general strongly real.

\begin{question}
Which characteristically simple Beauville groups are strongly real?
\end{question}

As a more specific conjecture on these matters we assert the following.


\begin{conjecture}\label{CharSimpConj}
If $H$ is a finite simple group of order greater than $3$, then the group $H\times
H$ is a strongly real Beauville group.
\end{conjecture}

Note that Corollary \ref{abCor} tells us that this is true for all
abelian characteristically simple Beauville groups. For the
nonabelian characteristically simple Beauville groups this
conjecture seems rather distant given that, at the time of writing,
we have neither a solution to Conjecture \ref{simpconj} nor do we even know if $H\times H$ for a simple group $H$ is even a Beauville group, let alone a strongly real one. Anyone
tempted to extend the above conjecture to the products of a larger
number of copies of simple groups should see the remarks following
Lemma \ref{Mathieu}, although some hope is provided by the results proven in Section \ref{alts}.

\begin{theorem}\label{thm}
Let $G$ be a strongly real Beauville group with coprime strongly
real Beauville structure $\{\{x_1,y_1\},\{x_2,y_2\}\}$. Furthermore, suppose
that there exists an automorphism $\phi\in\mbox{Aut}(G)$ such that
$$\phi(x_1)=x_1^{-1}\mbox{, }\phi(y_1)=y_1^{-1}\mbox{, }\phi(x_2)=x_2^{-1}\mbox{ and }\phi(y_2)=y_2^{-1}.$$
Then the group $G\times G$ is a strongly real Beauville group.
\end{theorem}

\begin{proof}
Consider the following elements of $G\times G$
$$g_1=(x_1,x_2)\mbox{, }h_1=(y_1,y_2)\mbox{, }g_2=(x_2,x_1)\mbox{ and }h_2=(y_2,y_1)$$
The pair $\{g_1,h_1\}$ generate the whole of $G\times G$ since the
elements $g_1^{o(x_2)}$ and $h_1^{o(y_2)}$ generate the first factor
whilst the elements $g_1^{o(x_1)}$ and $h_1^{o(y_1)}$ generate the
second factor thanks to our hypothesis that $o(x_1)o(x_2)o(x_1y_1)$
is coprime to $o(x_2)o(y_2)o(x_2y_2)$. Similarly $\langle
g_2,h_2\rangle= G\times G$.

We define an automorphism $\psi\in\mbox{Aut}(G\times G)$ such that for every $(g,h)\in G\times G$
$$\psi(g,h)=(\phi(g),\phi(h)).$$
This automorphism clearly makes the above Beauville structure for
$G\times G$ a strongly real Beauville structure.
\end{proof}

\begin{cor}
Conjecture \ref{CharSimpConj} is true for each of the following
groups.
\begin{enumerate}
\item[(a)] The alternating groups $A_n$ for $n\geq6$;
\item[(b)] The linear groups $PSL_2(q)$ for prime powers $q>5$;
\item[(c)] The Suzuki groups $^2B_2(2^{2n+1})$;
\item[(d)] All simple groups of order at most $100\,000\,000$;
\item[(e)] The sporadic simple groups.
\end{enumerate}

\end{cor}

\begin{proof}
For part (a) the results proved in Section \ref{alts} provide a strongly real Beauville structure for $A_n\times A_n$ for sufficiently large $n$, the smaller cases being straightforward calculations that are easily performed separately.

For part (b) we note that the strongly real Beauville structures
constructed by Fuertes and Jones in \cite{FuertesJones11} for the
groups $PSL_2(q)$ satisfy the hypotheses of Theorem \ref{thm}.

For part (c) we note that the strongly real Beauville structures for
$^2B_2(2^{2n+1})$ we constructed in Theorem \ref{Suz} are coprime
and satisfy the hypotheses of Theorem \ref{thm}.

For part (d) we note that the author's computations alluded to in
Section 2 were performed in such a way that the hypotheses of
Theorem \ref{thm} are satisfied.

Finally for part (e) we observe that for all the sporadic groups,
apart from the Mathieu groups M$_{11}$ and M$_{23}$, the structures
given by the author in \cite{FairbairnExceptional} satisfy the
hypotheses of Theorem \ref{thm}. The groups M$_{11}$ and M$_{23}$
are dealt with separately in Lemma \ref{Mathieu} below.
\end{proof}

We remark that the strongly real Beauville structures for the
quasisimple groups $SL_2(q)$ where $q>5$ constructed by Fuertes and
Jones in \cite{FuertesJones11} also satisfy the hypotheses of
Theorem \ref{thm} and so the groups $SL_2(q)\times SL_2(q)$ are also
strongly real.

Unfortunately, Theorem \ref{thm} cannot be applied to the strongly
real Beauville structures constructed by Fuertes and
Gonz\'{a}lez-Diez in \cite{FuertesGD10} for the symmetric and
alternating groups. This is because the types of the Beauville
structures in \cite{FuertesGD10} fail to satisfy the coprime
hypothesis since their structures use several elements of order 2. We return to this point in Section \ref{alts}.

Comparing the statement of Conjecture \ref{simpconj} with the
statement of Conjecture \ref{CharSimpConj} the reader should
immediately be asking ``what about the alternating group A$_5$ and the Matheiu groups M$_{11}$ and
M$_{23}$?" This concern is immediately addressed by the following
further piece of evidence for Conjecture \ref{CharSimpConj}.

\begin{lemma}\label{Mathieu}
The groups A$_5\times$A$_5$, M$_{11}\times$M$_{11}$ and M$_{23}\times$M$_{23}$ are
strongly real Beauville groups.
\end{lemma}

\begin{proof}
This is a straightforward computational calculation. Consider the following permutations.\\

\begin{tabular}{rclcrcl}
$x_1$&:=&(1,2,3,4,5)(6,7,8,9,10)&&$y_1$&:=&(2,3,4)(7,10)(6,9),\\
$x_2$&:=&(1,4,3,2,5)(7,8,9)&&$y_2$&:=&(1,2)(4,5)(6,9,8,7,10) and\\
$a$&:=&(1,5)(2,4)(6,10)(7,9)&&&&
\end{tabular}\vspace{3mm}

The set $\{\{x_1,y_1\},\{x_2,y_2\}\}$ gives a Beauville structure for the
group A$_5\times$A$_5$ of type ((5,6,5),(15,10,15)) acting intransitively on 5+5 points as a subgroup of the symmetric group S$_{10}$. The
automorphism $\alpha$ defined by conjugation by $a$ has the property
that
\[
\alpha(x_1)=x_1^{-1}\mbox{, }\alpha(y_1)=y_1^{-1}\mbox{, }\alpha(x_2)=x_2^{-1}\mbox{ and }\alpha(y_2)=y_2^{-1}
\]
from which we have that this Beauville structure is strongly real.

Next, the group M$_{11}\times$M$_{11}$. Consider the
following permutations.\\

\begin{tabular}{rcl}
$x_1$&:=&(1,2,3,4,5,6,7,8,9,10,11)(12,13,14,15,16,17,18,19,20,21,22)\\
$y_1$&:=&(1,6,10,5,2,7,4,9,11,8,3)(12,14,19,16,21,18,13,17,22,20,15)\\
$x_2$&:=&(1,3,9,11,10,7,2,4)(5,8)(12,14,20,22,19,21,16,13)(15,18)\\
$y_2$&:=&(2,6,9,4,8,3,7,5)(10,11)(12,13)(14,17,21,18,16,20,15,19)\\
$a$&:=&(1,22)(2,21)(3,20)(4,19)(5,18)(6,17)(7,16)(8,15)(9,14)(10,13)(11,12)\\
\end{tabular}\vspace{3mm}

The set $\{\{x_1,y_1\},\{x_2,y_2\}\}$ gives a Beauville structure for the
group M$_{11}\times$M$_{11}$ of type ((11,11,11),(8,8,8)) acting intransitively on 11+11 points as a subgroup of the symmetric group S$_{22}$. The
automorphism $\alpha$ defined by conjugation by $a$ has the property
that
\[
\alpha(x_1)=x_1^{-1}\mbox{, }\alpha(y_1)=y_1^{-1}\mbox{, }\alpha(x_2)=x_2^{-1}\mbox{ and }\alpha(y_2)=y_2^{-1}
\]
from which we have that this Beauville structure is strongly real.

Finally for M$_{23}\times$M$_{23}$ we similarly have that the permutations\\

\begin{tabular}{rcl}
$x_1$&:=&(1,22,5,17,6,10,18,16,19,8,9,15,13,14,21,4,3,7,23,20,2,12,11)\\
&&(24,40,44,43,26,33,34,32,38,39,28,31,29,37,41,30,42,25,46,36,35,45,27)\\
$y_1$&:=&(1,16,3,14,7,15,18,22,21,8,20,10,4,17,19,13,5,6,23,9,2,12,11)\\
&&(24,41,42,34,28,30,43,37,27,39,26,25,29,32,40,33,44,31,46,36,35,45,38)\\
$x_2$&:=&(1,3,19,7,18,4,11,21,16,14,6)(2,23,9,17,15,20,22,10,13,12,8)\\
&&(24,45,39,35,34,37,25,27,32,30,38)(26,36,43,29,40,28,44,46,41,33,31)\\
$y_2$&:=&(1,6,22,9,16,17,5,19,11,18,2)(3,14,13,23,4,12,15,10,7,21,8)\\
&&(24,34,33,44,39,26,40,37,32,35,43)(25,41,46,45,29,36,28,42,30,31,38)\\
$a$&:=&(1,46)(2,45)(3,44)(4,43)(5,42)(6,41)(7,40)(8,39)(9,38)(10,37)(11,36)\\
&&(12,35)(13,34)(14,33)(15,32)(16,31)(17,30)(18,29)(19,28)(20,27)\\
&&(21,26)(22,25)(23,24)\\
\end{tabular}\vspace{3mm}

\noindent define a strongly real Beauville structure of type
((23,23,23),(11,11,11)) for the group M$_{23}\times$M$_{23}$ acting intransitively on 23+23 points as a subgroup of the symmetric group S$_{46}$.
\end{proof}

We remark that in the examples of the above lemma, the automorphisms
used are outer automorphisms that interchange the two factors. The
lack of automorphisms that stop both M$_{11}$ and M$_{23}$ being
strongly real will therefore also stop the groups
M$_{11}\times$M$_{11}\times$M$_{11}$ and
M$_{23}\times$M$_{23}\times$M$_{23}$ being strongly real Beauville
groups. Furthermore it is easy to see that the permutations given in
the proof of Lemma \ref{Mathieu} can be adapted to construct a
strongly real Beauville structure of type ((88,88,88),(88,88,88))
for the group M$_{11}\times$M$_{11}\times$M$_{11}\times$M$_{11}$ and
to construct a strongly real Beauville structure of type
((253,253,253),(253,253,253)) for the group
M$_{23}\times$M$_{23}\times$M$_{23}\times$M$_{23}$. Similarly A$_5\times$A$_5\times$A$_5$ is not a strongly real Beauville group. It follows that
any extension of Conjecture \ref{CharSimpConj} to products of a
larger number of copies of simple groups will necessarily have a
much more complicated statement. It is likely that similar remarks apply to M$_{11}^{2k+1}$, M$_{11}^{2k}$, M$_{23}^{2k+1}$, M$_{23}^{2k}$ A$_5^{2k+1}$ and A$_5^{2k}$ for small values of $k$.

In light of the above it is natural to ask the following.

\begin{question}
Let $H$ be a finite simple group, $n\in\mathbb{Z}^+$ and $G=H^n$. When are inner automorphism sufficient to make $G$ strongly real and when do outer automorphisms interchanging the factors required? Moreover does this have any geometric significance for the corresponding surfaces?
\end{question}

\section{The Symmetric and Alternating Groups}\label{alts}

In the last section we discussed characteristically simple groups of the form $H\times H$ for some simple group $H$. In this section we prove slightly stronger results in the case of the alternating groups and a related result for the symmetric groups. In each of the below results conjugacy of elements is taken care of by the well known fact that two elements of the symmetric group are conjugate if, and only if, they have the same cycle type. We will use the following recent results of Jones.

\begin{lemma}\label{AltGenLem}
Let $H\leq S_n$.
\begin{enumerate}
\item[(a)] If $H$ is primitive and contains a cycle that fixes at least three points then $H\geq A_n$.

\item[(b)] If $H$ is transitive and contains an $m$-cycle where $m>n/2$ and $m$ is coprime to $n$ then $H$ is primitive.

\item[(c)] If $H$ is primitive and contains a cycle fixing two points then either $H\geq A_n$ or $PGL_2(q)\leq H\leq P\Gamma L_2(q)$ with $n=q+1$ for some prime power $q$.
\end{enumerate}
\end{lemma}

\begin{proof}
See \cite{j} and \cite[Section 6]{Jones1}.
\end{proof}

Before proving our main results we recall some facts about generating pairs in simple and characteristically simple groups. Let $H$ be a finite simple group. In \cite{Hall} Philip Hall showed that the largest $k$ such that the characteristically simple group $H^k$ is 2-generated is equal to the number of orbits of $Aut(H)$ on generating pairs of $H$. (He proved similar results for more general $n$-tuples but we will not be needing these results here.) To show that $H^k$ for some $k$ is generated by a pair of elements it is sufficient to show that each of the `coordinates' of these elements (that is, the parts of each permutation that correspond to each of the factors) are inequivalent under the action of $Aut(H)$. For $n\not=1,2,6$ we have that $Aut(A_n)=S_n$ and in the case $n=6$ the symmetric group S$_6$ is an index 2 subgroup of $Aut(\mbox{A}_6)\cong \mbox{P}\Gamma\mbox{L}_2(9)$.

\begin{lemma}
Let $n\geq11$ be odd and let $k\leq(n-6)/2$ be positive integers. Then $A_n^k$ is a strongly real Beauville group.
\end{lemma}

\begin{proof}
Since $A_n$ is simple for every $n>5$ we have that, by the remarks of the previous paragraph, it is sufficient to find $k$ pairs of generating pairs $T_{ij}=\{x_{ij},y_{ij}\}$ $i=1,2$, $j=1,\ldots,k$ such that for a fixed $i$ no $T_{ij}$ is an image of $T_{ij'}$ under the actions of automorphisms of $A_n$ for distinct $j$ and $j'$.

For our first pairs we set
$$x_{1j}=(1,\ldots,2j+3)\mbox{ and }y_{1j}=(2j+3,\ldots,n)$$
for $1\leq j<(n-6)/4$. These are cycles of odd length and are thus even permutations. Their product is an $n$-cycle. It is easy to check that $\langle x_{1j},y_{1j}\rangle$ is primitive and thus equal to $A_n$ since the group contains a cycle with at least three fixed points by Lemma \ref{AltGenLem}(a). These elements are both inverted by the automorphism defined by conjugation by
$$t=(1,2j)(2,2j-1)\cdots(j,j+1)(2j+2,n)(2j+3,n-1)\cdots((2j+n+1)/2,(2j+n+3)/2)$$
which has only one fixed point, namely $2j+1$.

For our second pairs we consider the permutations
$$x_{2j}=(1,\ldots,n-2)\mbox{ and }y_{2j}=(j+1,j+2)(n-j-1,n-j-2)((n-1)/2,n-1)((n+1)/2,n)$$
for $1\leq j<(n-5)/2$. These are again both even permutations, their product this time being an $(n-2)$-cycle. To confirm that these elements generate the group we note that $\langle x_{2j},y_{2j}\rangle$ is clearly transitive and so by Lemma \ref{AltGenLem}(b) must be primitive. It now follows from Lemma \ref{AltGenLem}(c) that $\langle x_{2j},y_{2j}\rangle=A_n$ since $n>9$ and the only elements of order 2 in $P\Gamma L_2(q)$ do not have the same cycle type as $y_{2j}$. (In the case $n=9$ these permutations generate $PSL_2(8)$.) These elements are both inverted by the automorphism defined by conjugation by
$$(2,n-2)(3,n-3)\cdots((n-1)/2,(n+1)/2)(n-1,n)$$
which has only one fixed point, namely 1, and thus differs from $t$ solely by an inner automorphism.
\end{proof}

\begin{lemma}
Let $n\geq12$ be an even integer and let $k\leq(n-8)/4$. Then $A_n^k$ is a strongly real Beauville group.
\end{lemma}

\begin{proof}
Again, we seek a collection of generating pairs that are not mapped to one another by automorphisms of $A_n$.

For our first pairs we consider the following elements.
$$x_{1j}=(1,\ldots,2j+5)\mbox{ and }y_{1j}=(n,\ldots,2j+5,2j+4)$$
for $1\leq j\leq(n-8)/4$. These are cycles of odd length and are thus even permutations. Their product is an $(n-1)$-cycle. It is easy to check that $\langle x_{1j},y_{1j}\rangle$ is primitive and thus equal to $A_n$ since the group contains a cycle with at least three fixed points by Lemma \ref{AltGenLem}(a). These elements are both inverted by the automorphism defined by conjugation by
$$t=(2j+4,2j+5)(1,2j+3)\cdots(j+1,j+3)(2j+6,n)\cdots((n+2j+4)/2-1,(n+2j+4)/2+1)$$
which has precisely two fixed points namely $j+2$ and $(n+2j+4)/2$.

For our second pair of generators we consider the permutations
$$x_{2j}=(1,\ldots,n-2)(n-1,n)\mbox{ and}$$
$$y_{2j}=(n/2,(n-2)/2,n-1)(j+1,j,n,n-j-1,n-j-2)$$
for $1\leq j<(n-2)/2$. These are both even permutations and their product is an $(n-3)$-cycle that fixes the points $j$, $n-j-2$ and $(n-2)/2$. The subgroup $\langle x_{2j}y_{2j},y_{2j}^3\rangle$ fixes the point $(n-2)/2$ and is transitive on the remaining $n-1$ points. It follows that the group $\langle x_{2j},y_{2j}\rangle$ is 2-transitive and is therefore primitive. Since the group $\langle x_{2j},y_{2j}\rangle$ also contains the 3-cycle $y_{2j}^5$ (and the 5-cycle $y_{2j}^3$) it is equal to $A_n$ by Lemma \ref{AltGenLem}(a). These elements are both inverted by the automorphism defined by conjugation by
$$(1,n-2)(2,n-3)\cdots((n-2)/2,n/2)$$
which has precisely two fixed points, namely $n-1$ and $n$ and thus differs from $t$ solely in an inner automorphism.
\end{proof}

When considering the alternating groups it is natural to seek similar results for the symmetric groups. Unfortunately, here we are somewhat limited: if $k>2$ then $S_n^k$ is not 2-generated since its abelianisation is $\mathbb{Z}_2^k$ and this is not 2-generated. It follows that we can only find analogous results for $k\leq2$ and since $k=1$ comes straight from the work of Fuertes and Gonz\'{a}lez-Diez we are left only to consider the case $k=2$. Note that in this case the outer automorphism of $S_n\times S_n$ that interchanges the two factors is useless since the only permutations that are inverted by this automorphism are even.

\begin{lemma}
For $n\geq5$ the group $S_n\times S_n$ is strongly real.
\end{lemma}

\begin{proof}
We will explicitly construct our Beauville structure in the case of $n$ even and then describe the differences in the case of $n$ odd.

For our first pair we consider the following elements.
$$x_1=(1,\ldots,n-1)(2n-1,2n)\mbox{ and }y_1=(n,n-1)(n+1,\ldots,2n-1)$$
The product of these permutations is a pair of $n$-cycles. It is easy to check that $\{x_1^2,y_1^{n-1}\}$ generates the first of the two factors whilst $\{x_1^{n-1},y_1^2\}$ generates the second and so $\langle x_1,y_1\rangle$ is the whole group. These elements are both inverted by the automorphism defined by conjugation by
$$t=(1,n-2)(2,n-3)\cdots(n/2,n/2-1)(n+1,2n-2)\cdots(3n/2,3n/2-1)$$
which has precisely four fixed points, namely $n-1$, $n$, $2n-1$ and $2n$.

For our second pair of generators we consider the permutations
$$x_2=(1,2,3)(4,\ldots,n)(n+4,n+3,n+2,n+1)$$
$$\mbox{ and }y_2=(4,3,2,1)(n+1,n+2,n+3)(n+4,\ldots,2n)$$
The product of these permutations is a pair of $n-2$ cycles. It is easy to check that $\{x_1^4,y_1^{3(n-3)}\}$ generates the first factor (this group is easily seen to be 2-transitive, and thus primitive, by considering conjugates of $y_1^{3(n-1)}$ under powers of $x_1^4$ and since this group contains a 4-cycle it contains the whole of $S_n$ by \cite[Corollary 1.3]{j}). Similarly $\{x_1^{3(n-1)},y_1^4\}$ generates the second factor. These elements are both inverted by the automorphism defined by conjugation by
$$(1,3)(5,n)\cdots(n/2+2,n/2+3)(n+1,n+3)(n+5,2n)\cdots(3n/2+2,n/2+3)$$
which has precisely four fixed points, namely $2$, $4$, $n+2$ and $n+4$ and thus differs from $t$ solely in an inner automorphism. (The case $n=6$ requires a little care --- using
$$x_2=(1,2,3,4)(10,11,12)\mbox{, }y_2=(4,5,6)(7,8,9,10),$$ and the automorphism defined by $(1,3)(5,6)(7,9)(11,12)$ avoids being trapped inside copies of $S_5$ acting transitively on six points.)

If $n$ is odd then for the first pair we need only replace the $(n-1)$-cycles with $n$-cycles to ensure that the elements have odd parity and for the automorphism instead use
$$t=(n,n-1)(n-1,1)\cdots((n-1)/2-1,(n-1)/2+1)\cdots$$
$$(2n,2n-1)(2n-2,n+1)\cdots((3n-1)/2-1,(3n-1)/2+1).$$
For the second generating pair we must now replace the $4$-cycles with some longer $p$-cycle whose length is coprime to $3(n-3)$ (if 5 fails then $n\geq2\times5+3$ and we can try 7; if both 5 and 7 fail then $n\geq2\times5\times7+3$ and we can try 11 etc.). The permutation of order 2 for the automorphism needs to be adjusted in the obvious manner (i.e. $(1,3)(4,p)(5,p-1)\cdots(p+1,n)(p+2,n-1)\cdots$). Again, the smallest case needs separate attention but it is easily checked that if
$$x_1=(1,4)(2,5)(6,10)(7,8,9)\mbox{ and }y_1=(1,5)(2,3,4)(6,9)(7,10)$$
and
$$x_2=(1,2,3,4,5)(6,7,9,10)\mbox{ and }y_2=(5,4,2,1)(10,9,8,7,6)$$
then $\{\{x_1,y_1\},\{x_2,y_2\}\}$ is a strongly real Beauville structure whose elements are inverted by conjugation by the element $(1,5)(2,4)(6,10)(7,9).$
\end{proof}

\section{Almost Simple Groups}\label{AS}

Let $G$ be a group. Recall that we say $G$ is almost simple if there exists a simple group $S$ such that $S\leq G\leq\mbox{Aut}(S)$. For example, any simple group is almost simple, as are the symmetric groups. Given our earlier remarks on the finite simple groups it is natural to ask the following.

\begin{question}
Which of the almost simple groups are strongly real?
\end{question}

This is particularly pertinent in light of Fuertes and Gonz\'{a}lez-Diez proof that the symmetric groups $S_n$ for $n>5$ are strongly real. Unfortunately, the general picture here is much more complicated with many almost simple groups not even being Beauville groups, let alone strongly real Beauville groups. Worse, infinitely many of the almost simple groups are not even 2-generated: the smallest example is PSL$_4(9)$ whose outer automorphism group is $\mathbb{Z}_2\times\mbox{D}_8$ (and more generally, if $p$ is an odd prime and $r$ is an even positive integer then $Aut(PSL_4(p^r))$ is not 2-generated). We can at least add the following to the list.

\begin{theorem}\label{spor}
The non-simple almost simple sporadic groups are strongly real Beauville groups.
\end{theorem}

Before proceeding to the proof of Theorem \ref{spor} we make the following remarks for those who are unfamiliar with standard generators of finite groups (those who are familiar with them may skip to the proof in the next paragraph). Any given group will have many generating sets and in particular if $x,y\in G$ are such that $\langle x,y\rangle=G$ then $\langle x^g,y^g\rangle=G$ for any $g\in G$. To provide some standardisation to computational group theory, Wilson \cite{Wilson} introduced the notion of `standard generators'. These are generators for a group that are unusually easy to find and are specified in terms of which conjugacy classes that they belong to (which can often be determined solely from their orders) and which classes some word(s) in these elements belong to. Representatives for many of the finite simple groups and various other groups closely related to them may be found in explicit permutations and/or matrices for many of their most useful representations on the web-based Atlas of Group Representations \cite{elecATLAS}.

To construct our Beauville structures that prove Theorem \ref{spor} we proceed as follows. We first recall some well-known facts about the sporadic simple groups. If $G$ is one of the 27 sporadic simple groups (including the Tits group $^2$F$_4(2)'$) then the outer automorphism group of $G$ has order at most 2 and that in all cases in which there exists a non-trivial outer automorphism $Aut(G)$ is a non-split extension, apart from the Tits group $^2$F$_4(2)'$ and may thus be written $G:2$ in {\sc ATLAS} notation (see Section 1). Let $G$ be a simple group such that Aut$(G)=G:2$. Let $t,t'\in G:2$ have order 2 such that one of the elements lies in $G$ and the other lies in $G:2\setminus G$. For $i=1,2$ we define the elements $x_i=tt'^{g_i}$ for some $g_i\in G:2$. If for $i=1,2$ $u(i)\in C_G(t)$, then we can further define the elements $y_i=(x_i^{j(i)})^{u(i)}$ for some positive integers $j(i)$. Note that since $u(i)$ commutes with $t$ the automorphism defined by conjugation by $t$ inverts both $x_i$ and $y_i$. Using knowledge of the subgroup structure of $G:2$ it is often possible to choose the elements $g_1$, $g_2$, $u(1)$ and $u(2)$ in such a way that $\langle x_1,y_1\rangle=\langle x_2,y_2\rangle=G:2$. Unfortunately, in the case of almost simple groups we must have that the orders of $x_i$ and $y_i$ all have even order and so verifying the conjugacy condition ($\dagger$) of Definition \ref{DagDef} is more difficult than simply showing that $o(x_1)o(y_1)o(x_1y_1)$ is coprime to $o(x_2)o(y_2)o(x_2y_2)$. For some of the larger groups verifying that they generate the whole group can also be difficult. In these cases, generation is verified by finding words in our elements with the property that no proper subgroup can contain them (in many cases the maximal subgroups for these groups may be found in \cite{ATLAS}). The words defining our Beauville structures are given in Table \ref{Table1} and their types are given in Table \ref{Table2}.

We remark that this construction will not work in cases that are non-split extensions. This includes the almost simple `sporadic' Tits group $^2$F$_4(2)$. Straightforward computations verify that this group is not a strongly real Beauville group.

In light of the above we make the following tentative conjecture.

\begin{conjecture}
A split extension of a simple group is a Beauville group if, and only if, it is a strongly real Beauville group.
\end{conjecture}

We remark that some (unpublished) progress on this conjecture has been made by the author's PhD student, Emilio Pierro, whilst the question of which of the groups $PGL_2(q)$ are Beauville is discussed by Garion in \cite{Garion}.

\begin{table}
\begin{center}
\begin{tabular}{|c|c|c|c|c|}%
\hline
$G$&$t_1$&$t_2$&$x_1$&$x_2$\\
\hline\hline
M$_{12}:2$&$c$&$(cd)^6$&$t_1t_2$&$t_1t_2^d$\\
M$_{22}:2$&$((cd)^2d)^5$&$d^2$&$t_1t_2$&$t_1t_2^c$\\
J$_2:2$&$c$&$(cd^2(cd)^2)^6$&$t_1t_2$&$t_1t_2^{d^4}$\\
HS$:2$&$c$&$((cd)^3cd^2)^5$&$t_1t_2$&$t_1t_2^d$\\
J$_3:2$&$c$&$(cd)^{12}$&$t_1t_2$&$t_1t_2^d$\\
McL$:2$&$c$&$((cd)^2(cd^2)^2(cd)^2d)^2$&$t_1t_2$&$t_1t_2^{(dcd)^2}$\\
He$:2$&$c$&$d^3$&$t_1t_2$&$t_1t_2^{cd(cd^2)^2c}$\\
Suz$:2$&$c$&$(cd)^{14}$&$t_1t_2$&$t_1t_2^{(dc)^2(d^2c)^2d^2}$\\
O'N$:2$&$c$&$d^2$&$t_1t_2$&$t_1t_2^{cd}$\\
Fi$_{22}:2$&$(cd^4)^{10}$&$(cd^3)^{15}$&$t_1t_2^{dcd^6}$&$t_1t_2^{dcd}$\\
HN$:2$&$c$&$(cd^3(cd)^2)^{12}$&$t_1t_2^{(dcd)^2d^2}$&$t_1t_2^{dcd^4cd^2}$\\
Fi$_{24}$&$d^4$&$((cd)^2d^3)^33$&$t_1t_2^{d^4c}$&$t_1t_2^{dcd^2c}$\\
\hline\hline
$G$&$u_1$&$u_2$&$j(1)$&$j(2)$\\
\hline\hline
M$_{12}:2$&$[c,(dc)^2d^2]^3$&$(dc)^2d[c,(dc)^2d]^2$&1&1\\
M$_{22}:2$&$cd^2cd[t_1,cd^2cd]^5$&$[t_1,c]^2$&5&9\\
J$_2:2$&$d[c,d]^3$&$d[c,d]^3$&1&9\\
HS$:2$&$d[c,d]$&$d[c,d]$&1&1\\
J$_3:2$&$d[c,d]^4$&$d[c,d]^4$&21&1\\
McL$:2$&$d^2[c,d^2]^7$&$dcd[c,dcd]^7$&1&7\\
He$:2$&$d[c,d]^7$&$d[c,d]^7$&15&19\\
Suz$:2$&$d[c,d]^3$&$d[c,d]^3$&9&3\\
O'N$:2$&$[c,d]^5$&$[c,d]^5$&7&1\\
Fi$_{22}:2$&$[t_1,d^3]^3$&$[t_1,d]^3$&3&1\\
HN$:2$&$d^2[cd^2]^10$&$d[c,d]^4$&1&39\\
Fi$_{24}$&$c[t_1,c]$&$c[t_1,c]$&7&25\\
\hline
\end{tabular}\caption{Words in the standard generators providing strongly real Beauville structures for each of the non-simple almost simple sporadic groups.}\label{Table1}
\end{center}
\end{table}

\begin{table}
\begin{center}
\begin{tabular}{|c|c|c|c|}
\hline
$G$&type&$G$&type\\
\hline\hline

M$_{12}:2$&((4,4,5),(6,6,3))&He : 2&((16,16,7),(30,30,5))\\
M$_{22}:2$&((12,12,4),(10,10,5))&Suz : 2&((10,10,3),(8,8,13))\\
J$_2:2$&((24,24,15),(14,14,7))&O'N : 2&((38,38,19),(56,56,28))\\
HS : 2&((8,8,8),(6,6,15))&Fi$_{22}$ : 2&((10,10,11),(12,12,4))\\
J$_3:2$&((34,34,17),(24,24,4))&HN : 2&((18,18,25),(44,44,22))\\
McL : 2&((8,8,3),(10,10,5))&Fi$_{24}$&((66,66,33),(84,84,26))\\

\hline
\end{tabular}\caption{The types of the Beauville structures specified by the words in Table \ref{Table1}.}\label{Table2}
\end{center}
\end{table}

\section{Nilpotent Groups}

It is immediate that the direct product of two Beauville groups of
coprime order is again a Beauville group (though slightly more is
true --- see \cite[Lemma 1.3]{BarkerBostonFairbairn}). Recall that a
finite group is nilpotent if and only if it is the direct product of
its Sylow subgroups. Since Sylow subgroups for different primes will
have coprime orders this observation reduces the study of
nilpotent Beauville groups to that of Beauville $p$-groups.

There is another motivation for wanting to study Beauville
$p$-groups and that is to study how finite groups in general behave
from the point of view of Beauville constructions. We saw in Section
2 that among the non-abelian finite simple groups only one fails to
be a Beauville group and of the rest only two fail to be strongly
real. This immediately raises the following question.
\begin{question}
Are most Beauville groups strongly real Beauville groups?
\end{question}
Finite simple groups are rare gems in the rough --- for every positive
integer $n$ there are at most two finite simple groups of order $n$
and for most values of $n$ there are none at all. Taking our lead
from their behaviour is therefore somewhat dangerous.

Few mathematicians outside finite group theory seem to realise that
in some sense most finite groups are $p$-groups, indeed most finite
groups are 2-groups. There are $49\,910\,529\,484$ groups of order
at most $2\,000$. Of these $49\,487\,365\,422$ have order precisely
$1\,024$ - that's more than $99\cdot1\%$ of the total! When we throw
in the other 2-groups of order at most $1\,024$ and the other
$p$-groups of order at most $2\,000$ we have almost all of them.
Determining which of the Beauville $p$-groups are strongly real
Beauville $p$-groups thus goes a long way to answering the above
question for groups in general. (For details of these extraordinary
computational feats and a historical discussion of the problem of
enumerating groups of small order, which has been worked on for
almost a century and a half, see the work of Besche, Eick and
O'Brian in \cite{obrian1,obrian2}.)

Theorem \ref{ab} and Corollary \ref{abCor} tell us that if $p\geq5$
is prime then there are infinitely many strongly real Beauville
$p$-groups - just let $n$ be any power of $p$. These results are,
however, useless for the primes 2 and 3. As far as the author is
aware there are no known examples.

\begin{problem}
Find strongly real Beauville $2$-groups and $3$-groups.
\end{problem}

The only known infinite family of
Beauville 2-groups are those recently constructed by Barker, Boston,
Peyerimhoff and Vdovina in \cite{Barkeretal13}. One of the main
results of \cite{Barkeretal13} is that the groups constructed there
are not strongly real. Furthermore there remain only
finitely many known examples of Beauville 3-groups.

In general, $p$-groups have large outer automorphism groups
\cite{Boston,BostonBushHajir}, so it seems likely that most
Beauville $p$-groups are in fact strongly real. Again, as far as the
author is aware, this matter remains largely uninvestigated.
\begin{problem}
Find non-abelian strongly real Beauville $p$-groups.
\end{problem}

The best general discussion of work on Beauville $p$-groups is
Boston's contribution to these proceedings \cite{BostonSurv}. The work of the Barker,
Boston and the author in \cite{BarkerBostonFairbairn} and the work
of Barker, Boston, Peyerimhoff and Vdovina in
\cite{Barkeretal13,Barkeretal11,BarkeretalAgain} are also worth
consulting.\\

\textbf{Acknowledgements} The author wishes to express his deepest gratitude to the organisers of the Beauville Surfaces and Groups 2012 conference held in the University of Newcastle without which this volume, and thus the opportunity to present these results here would not have been possible. The author also wishes to thank Professor Gareth Jones for many invaluable comments on earlier drafts of this article, particularly regarding the results concerned with products of symmetric and alternating groups. Finally, the author wishes to thank the anonymous referee whose comments and suggestions have substantially improved the readability of this paper, particularly bearing in mind the wide breadth of the audience for this work.

\end{document}